\documentclass[reqno,12pt]{amsart}

\usepackage{amsmath, amsthm, amsfonts, amssymb,color}
\input{mathrsfs.sty}

\usepackage{tikz}
\usetikzlibrary{arrows.meta}

\usepackage{amsmath}
\usepackage{amsfonts}
\usepackage{amssymb}
\usepackage{eufrak}
\usepackage{amscd}
\usepackage{amsthm}
\usepackage{amstext}
\usepackage[all]{xy}

 \newcommand{\e}{{\varepsilon}}

\newcommand{\supp}{\operatorname{supp}}

\newcommand{\diam}{\operatorname{diam}}

\newcommand{\Z}{\mathbb Z/2\mathbb Z}

   \theoremstyle{plain}
   \newtheorem{thm}{Theorem}
   
   \newtheorem{lem}[thm]{Lemma}
   \newtheorem{cor}[thm]{Corollary}
   \theoremstyle{definition}
   
   \newtheorem{defn}[thm]{Definition}

	\newtheorem{example}{Example}
	\newtheorem{remark}[thm]{Remark}
	
   \theoremstyle{remark}

\author{V. Manuilov}

\date{}

\address{Moscow Center for Fundamental and Applied Mathematics, Moscow State University,
Leninskie Gory 1, Moscow, 
119991, Russia}

\email{manuilov@mech.math.msu.su}

\title {Paths in graphs: bounded geometry and property A}

\begin{document}

\begin{abstract}
We expose a class of discrete metric spaces, for which bounded geometry is equivalent to the property A of G. Yu. This class includes the coarse disjoint union of $(\Z)^n$, $n\in\mathbb N$, and consists of spaces of simple paths in a class of graphs that includes cactus graphs, with the metric defined as the number of edges in the symmetric difference of the paths. We also show that if a space in this class does not have bounded geometry then it contains a subspace of bounded geometry without property A.

\end{abstract}

\maketitle

\section*{Introduction}

Discrete metric spaces of bounded geometry are central to coarse geometry, primarily because their Rips complexes are finite-dimensional, which frequently simplifies arguments. Notably, within this class, many definitions that yield different results for general discrete spaces become equivalent. Key examples of spaces with bounded geometry include all finitely generated groups endowed with a word length metric, and, more generally, all uniformly finite graphs.

Property A, a concept introduced by G. Yu \cite{Yu} in the context of coarse geometry, serves as a coarse version of amenability. Its importance stems from the fact that it guarantees coarse embeddability into a Hilbert space --- a property crucial for problems surrounding the Baum-Connes conjecture, even though property A is a strictly stronger condition.

It is known that bounded geometry and property A are independent. For instance, every locally finite tree (with its graph metric) satisfies property A, irrespective of whether it has bounded geometry. (Recall: local finiteness requires each vertex to have finitely many neighbors, while bounded geometry demands a uniform bound on this number across all vertices). Furthermore, among spaces that do have bounded geometry, one can find examples both with and without property A. To complete the picture, we should mention the Nowak's example \cite{Nowak} of a coarse disjoint union of the groups $(\mathbb Z/2\mathbb Z)^n$, $n\in\mathbb N$, which possesses neither bounded geometry nor property A.

Nowak's example of a metric space without property A constructed from finite abelian groups shows that the absense of bounded geometry may be one of the reasons for the failure of property A. The two questions arise: \emph{can absense of bounded geometry be the sole reason for the failure of property A?} and \emph{is there a metric space without property A such that every its subspace of bounded geometry has property A?} Nowak's example would be a natural candidate for the second question, but it follows from \cite{Ostrovsky} that it contains subspaces of bounded geometry without property A. Being unable to answer the second question, we give a partial answer to the first question and expose a natural class of discrete metric spaces, for which bounded geometry is equivalent to property A. This specific class includes the Nowak's example and consists of the spaces of simple paths in the class of graphs that we call graphs with controlled cycles, with the metric that counts edges in the symmetric difference of paths. In particular, this class of graphs includes the cactus graphs.

\section{Basics from coarse geometry}

Our reference on coarse geometry is \cite{Nowak-Yu}.
Let $X$ be a discrete metric space with the metric $d_X$. We write $B_r(x)$ for the open ball of radius $r$ centered at $x\in X$, $B_r(x)=\{y\in X:d_X(x,y)<r\}$. If $Y\subset X$ is a finite subset then we denote by $|Y|$ the number of points in $Y$. Recall that $X$ is \emph{uniformly discrete} if $\inf_{x\neq y}d_X(x,y)>0$ and it has \emph{bounded geometry} if for any $r>0$ there exists $C_r$ such that $|B_r(x)|\leq C_r$ for any $x\in X$. 

A discrete space $X$ has \emph{property A} if for any $R>0$, $\epsilon>0$ there exists $S>0$ and finite subsets $\{A_x\subset X\times\mathbb N:x\in X\}$ such that 
\begin{enumerate}
\item
$\frac{|A_x\Delta A_y|}{|A_x\cap A_y|}<\epsilon$ if $d_X(x,y)\leq R$;
\item
$A_x\subset B_S(x)\times\mathbb N$ for any $x\in X$.
\end{enumerate} 

Note that it was shown in \cite{Zhu} that, for bounded geometry spaces, one can get rid of the seemingly redundant factor $\mathbb N$ and to use subsets $A_x\subset X$, $x\in X$, in the above definition. 

A discrete space $X$ satisfies the \emph{Higson--Roe condition} if for any $R>0$, $\epsilon>0$ there exists $S>0$ and a map $\xi:X\to l_2(X)$ such that 
\begin{itemize}
\item[(0)]
$\|\xi(x)\|=1$ for any $x\in X$;
\item[(1)]
$\|\xi(x)-\xi(y)\|<\epsilon$ if $d_X(x,y)\leq R$; 
\item[(2)]
$\supp\xi(x)\subset B_S(x)$ for any $x\in X$.
\end{itemize} 

Property A implies the Higson--Roe condition, and if $X$ is of bounded geometry then they are equivalent (\cite{Nowak-Yu}, Section 4.2).

Given two uniformly discrete metric spaces, $(X,d_X)$ and $(Y,d_Y)$, a {\it coarse equivalence} between them is a map $f:X\to Y$ such that there exists a self-homeomorphism $\varphi$ of $[0,\infty)$ with $\varphi^{-1}(d_X(x,x'))<d_Y(f(x),f(x'))<\varphi(d_X(x,x'))$ for any $x,x'\in X$. 

For a sequence of metric spaces $(X_n,d_n)$, $n\in\mathbb N$, their {\it coarse disjoint union} is the space $\sqcup_{n\in\mathbb N} X_n$ with a metric $d$ such that on each $X_n$ $d$ coincides with $d_n$, and the distance between points in different $X_n$'s can be defined by the formula 
$$   
d(y_n,y_m)=d_n(y_n,x_n)+d_m(y_m,x_m)+\alpha_{nm},\quad n\neq m,
$$
where $x_i\in X_i$ are fixed points, $y_i\in X_i$, and $\alpha_{nm}$ satisfies 
\begin{equation}\label{e2}
\lim_{n,m\to\infty,n\neq m}\alpha_{nm}=\infty.
\end{equation} 
Clearly, $d$ satisfies the triangle inequality.

\begin{lem}
Different choice of the fixed points and of the numbers $\alpha_{nm}$ gives a metric that is coarsely equivalent to $d$, hence the coarse disjoint union is well defined up to coarse equivalence. 

\end{lem}
\begin{proof}
Let $x'_i\in X_i$, $d_i(x_i,x'_i)=\beta_i$, and let $\alpha'_{nm}$ satisfies (\ref{e2}). Set
$$   
d'(y_n,y_m)=d_n(y_n,x'_n)+d_m(y_m,x'_m)+\alpha'_{nm},\quad n\neq m,
$$
and for shortness' sake write $d_n(y_n,x_n)+d_m(y_m,x_m)=t$. It follows from (\ref{e2}) that there exists a self-homeomorphism $\varphi$ of $[0,\infty)$ such that $\alpha'_{nm}+\beta_n+\beta_m<\varphi(\alpha_{nm})$ for any $n,m\in\mathbb N$, $n\neq m$. Then
$$
d'(y_n,y_m)\leq t+\beta_n+\beta_m+\alpha'_{nm}<t+\varphi(\alpha_{nm})<\psi(t+\alpha_{nm})=\psi(d(y_n,y_m)),
$$
where $\psi(x)=\varphi(x)+x$. Interchanging $d$ and $d'$, we see that they are coarsely equivalent.
\end{proof}

When all $X_n$ are bounded, our definition of the coarse disjoint union coincides with the standard one (\cite{Nowak-Yu}, Example 1.4.12).

\section{Basics from graph theory}

Let $\Gamma=(V,E)$ be a graph with the set $V=V(\Gamma)$ of vertices and $E=E(\Gamma)$ of edges. We assume that it is undirected, has no self-loops, but may have  multiple edges between two vertices. 

A path in $\Gamma$ is a sequence $(e_1,\ldots,e_n)$ of edges such that each $e_i$ connects the vertices $v_{i-1}$ and $v_i$, $i=1,\ldots,n-1$. Then we may write this path also as a sequence of vertices, $[v_0,v_1,\ldots,v_n]$.
A graph is connected if any two vertices can be connected by a path. The length of a path is the number of edges. This allows to define a metric $d_\Gamma$ on the set $V$ of vertices of a connected graph as the minimal length of paths connecting the two given vertices. A path $[v_0,v_1,\ldots,v_n]$ is a {\it cycle} if all $v_i$, $i=0,1,\ldots,n-1$ are different and $v_n=v_0$. A path $[v_0,v_1,\ldots,v_n]$ is {\it simple} if all $v_i$, $i=0,1,\ldots,n$ are different. Note that in this case, each edge is passed not more than once. If a graph is connected then each two different vertices can be connected by a simple path. 
A graph is {\it rooted} if one vertex, say $v_0$, is fixed. Let $S(\Gamma)$ denote the set of all simple paths in $\Gamma$ starting at the root $v_0$, and let $P_v(\Gamma)\subset S(\Gamma)$ be its subset of all simple paths with the endpoint at $v\in\Gamma$.  

\begin{lem}\label{inject}
Let $\gamma,\delta\in S(\Gamma)$, $\gamma=(e_1,\ldots,e_n)$, $\delta=(e'_1,\ldots,e'_n)$, and let $\gamma$ and $\delta$ consist of the same edges. Then $\gamma=\delta$. 

\end{lem}
\begin{proof}
Suppose that $e'_1\neq e_1$. Note that both $e_1$ and $e'_1$ start at $v_0$. Clearly, $e_1$ is passed by $\delta$ after $e'_1$, and, as $v_0$ is one of the endpoints of $e_1$, $\delta$ should pass through $v_0$ for the second time --- a contradiction with simplicity. Thus, $e'_1=e_1$. Inductively, $e'_i=e_i$ for any $i=1,\ldots,n$.  
\end{proof}

By Lemma \ref{inject}, we can, and will, view paths in $S(\Gamma)$ as subsets of $E(\Gamma)$.

A connected graph is a {\it cactus} graph if any two cycles have no more than one common vertex (hence, no common edges). Our reference for graphs, in particular for cactus graphs, is \cite{graphs}.

A graph $\Gamma'=(V',E')$ is a subgraph of $\Gamma=(V,E)$ if $V'\subset V$ and $E'$ consists of all edges from $E$ such that both their endpoints are in $V'$.

\section{Graphs with controlled cycles}

Let $C$ be a cycle in a connected graph $\Gamma$. We say that a path $\gamma$ {\it intersects} $C$ if they have at least one common edge, i.e. intersect as sets of edges, $\gamma\cap C\neq \emptyset$. 
We say that two paths, $\gamma$ and $\delta$, {\it diverge on $C$} if  they both intersect $C$, and $(\gamma\cap C)\sqcup(\delta\cap C)=C$, where $\sqcup$ denotes the disjoint union of edges.  
We write $\gamma_C$ for the path that coincides with $\gamma$ on $\gamma\setminus C$, and then goes through $C\setminus\gamma$ instead of $C\cap\gamma$. 
Note that $\gamma$ and $\gamma_C$ diverge on $C$.
The picture below illustrates these notions.

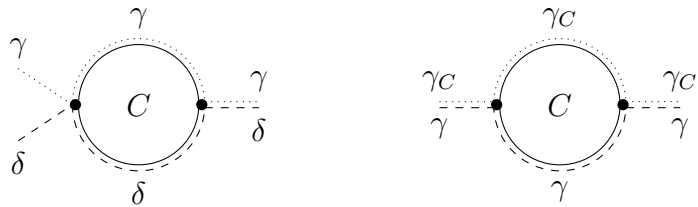
\begin{center}

\begin{tikzpicture}[scale=0.8]

\draw (2,0) circle (1cm); 
\draw[dotted] (3.1,0) arc (0:180:1.1cm);
\draw[dashed] (0.9,0) arc (180:360:1.1cm);

\node at (2,0) {$C$};

\node[circle, fill=black, inner sep=1.5pt] at (0.95,0) {};
\node[circle, fill=black, inner sep=1.5pt] at (3.05,0) {};

\node[above] at (0,0.6) {$\gamma$};
\node[below] at (0,-0.6) {$\delta$};

\draw[dotted] (0,0.6) -- (0.9,0);
\draw[dashed] (0,-0.6) --  (0.9,0);

\draw[dotted] (3.1,0.05) -- (4,0.05);
\draw[dashed] (3.1,-0.05) -- (4,-0.05);
 
\node[above] at (4,0.05) {$\gamma$};
\node[below] at (4,-0.05) {$\delta$};

\node[above] at (2,1.1) {$\gamma$};
\node[below] at (2,-1.1) {$\delta$};


\draw (9,0) circle (1cm); 
\draw[dotted] (10.1,0) arc (0:180:1.1cm);
\draw[dashed] (7.9,0) arc (180:360:1.1cm);

\node at (9,0) {$C$};

\node[circle, fill=black, inner sep=1.5pt] at (7.95,0) {};
\node[circle, fill=black, inner sep=1.5pt] at (10.05,0) {};

\node[above] at (7,0.05) {$\gamma_C$};
\node[below] at (7,-0.05) {$\gamma$};

\draw[dotted] (7,0.05) -- (7.9,0.05);
\draw[dashed] (7,-0.05) --  (7.9,-0.05);

\draw[dotted] (10.1,0.05) -- (11,0.05);
\draw[dashed] (10.1,-0.05) -- (11,-0.05);
 
\node[above] at (11,0.05) {$\gamma_C$};
\node[below] at (11,-0.05) {$\gamma$};

\node[above] at (9,1.1) {$\gamma_C$};
\node[below] at (9,-1.1) {$\gamma$};

\end{tikzpicture}

\end{center}
  
\begin{defn}
Let $\gamma$, $\delta$ be simple paths in $\Gamma$ with the same endpoints. We write $\gamma\sim_n\delta$ if $\gamma$ and $\delta$ do not diverge on any cycle of length $>n$. We call $\Gamma$ a {\it graph with controlled cycles} if $\sim_n$ is an equivalence relation on the set of simple paths of $\Gamma$ with the fixed endpoints for any $n\in\mathbb N$.

\end{defn}

\begin{thm}\label{equiv_rel}
If $\Gamma$ is a cactus graph then $\sim_n$ is an equivalence relation on the set of simple paths of $\Gamma$ with the fixed endpoints for any $n\in\mathbb N$.

\end{thm}
\begin{proof}
We have to check transitivity of $\sim_n$.
Suppose that $\gamma\sim_n\delta$ and $\delta\sim_n\e$, but $\gamma\nsim_n\e$. Then there exists a cycle $C$ of length $>n$ such that $C=(\gamma\cap C)\sqcup(\e\cap C)$. This implies that $\gamma\cap\e\cap C=\emptyset$ (i.e. $\gamma$, $\delta$ and $C$ have no common edges), therefore, $C\subset\gamma\Delta\e$. As $\gamma\Delta\e=(\gamma\Delta\delta)\Delta(\delta\Delta\e)$, we have that $C\subset (\gamma\Delta\delta)\Delta(\delta\Delta\e)$. 

Note that, as all paths have the same endpoints, both $\gamma\Delta\delta$ and $\delta\Delta\e$ consist of some cycles, and as all cycles are isolated, these cycles have no common edges, hence each of these cycles, including $C$, should lie either in $\gamma\Delta\delta$ or in $\delta\Delta\e$. If $C\subset\gamma\Delta\delta$ then $\gamma\nsim_n\delta$, and if $C\subset\delta\Delta\e$ then $\delta\nsim_n\e$. This contradiction finishes the proof.
\end{proof}

\begin{example}
Here is a graph that is not a graph with controlled cycles.

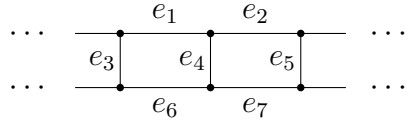
\begin{center}

\begin{tikzpicture}[scale=1.2]

\node[circle, fill=black, inner sep=1pt] at (1,0) {};
\node[circle, fill=black, inner sep=1pt] at (1,0.6) {};

\node[circle, fill=black, inner sep=1pt] at (2,0) {};
\node[circle, fill=black, inner sep=1pt] at (2,0.6) {};

\node[circle, fill=black, inner sep=1pt] at (3,0) {};
\node[circle, fill=black, inner sep=1pt] at (3,0.6) {};

\draw (1,0) -- (1,0.6); 
\draw (2,0) -- (2,0.6); 
\draw (3,0) -- (3,0.6); 

\draw (0.5,0) -- (1,0);
\draw (1,0) -- (2,0);
\draw (2,0) -- (3,0);
\draw (3,0) -- (3.5,0);

\draw (0.5,0.6) -- (1,0.6);
\draw (1,0.6) -- (2,0.6);
\draw (2,0.6) -- (3,0.6);
\draw (3,0.6) -- (3.5,0.6);

\node at (0,0) {$\cdots$}; 
\node at (0,0.6) {$\cdots$}; 
\node at (4,0) {$\cdots$}; 
\node at (4,0.6) {$\cdots$};

\node[above] at (1.5,0.6) {$e_1$};
\node[above] at (2.5,0.6) {$e_2$};

\node[below] at (1.5,0) {$e_6$};
\node[below] at (2.5,0) {$e_7$};

\node at (0.8,0.3) {$e_3$};
\node at (1.8,0.3) {$e_4$};
\node at (2.8,0.3) {$e_5$};

\end{tikzpicture}

\end{center}

\noindent
Consider the paths $\gamma=(e_1,e_2,e_5)$, $\delta=(e_1,e_4,e_7)$ and $\e=(e_3,e_6,e_7)$. Then $\gamma$ and $\delta$, and also $\delta$ and $\e$, diverge on a cycle of length 4, and do not diverge on the cycle of length 8, while $\gamma$ and $\e$ diverge on a cycle of length 8.

\end{example}

In graphs with controlled cycles, cycles may have common edges, but if the common part $C_1\cap C_2$ of $C_1$ and $C_2$ is connected, and, hence, $C_3=C_1\Delta C_2$ is a cycle then the two greatest numbers among $|C_1|$, $|C_2|$ and $|C_3|$ should coincide.

\begin{example}\label{example}
Let $\mathbf k=(k_i)_{i=0}^\infty$ be a sequence of integers with $0=k_0<k_1<k_2<\cdots$, and let $l\in\mathbb N$, $l\geq 2$. Let $\Gamma_{\mathbf k,l}$ be a planar graph with the vertices $(k_i,0)$, $i\in\mathbb N$, and $(n,j)$, where $n\neq k_0,k_1,k_2,\ldots$, $j=0,1,2,\ldots,l$. If $k_{i+1}=k_i+1$ then connect $(k_i,0)$ and $(k_{i+1},0)$ by $l$ edges. Otherwise, connect $(k_i,0)$ with each $(k_i+1,j)$ by one edge, and connect $(k_{i+1},0)$ with each $(k_{i+1},0)$ with each $(k_{i+1}-1,j)$ by one edge, $j=0,1,2,\ldots,l$. Also if $k_i<n,n+1<k_{i+1}$ then connect $(n,j)$ and $(n+1,j)$ by one edge for each $j=0,1,2,\ldots,l$. The graphs $\Gamma_{\mathbf k,l}$ have isolated cycles for any $\mathbf k,l$. Here is the picture of $\Gamma_{\mathbf k,l}$ for $\mathbf k=(0,1,3,7,8,\ldots)$ and $l=3$:

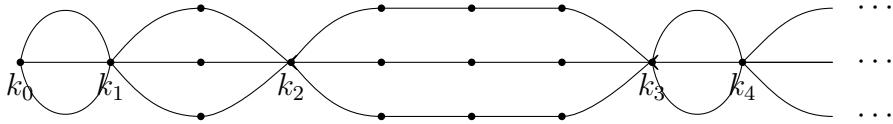
\begin{center}

\begin{tikzpicture}[scale=1.2]

\foreach \x in {0,1,...,8} {
    \node[circle, fill=black, inner sep=1pt] at (\x,0) {};
}
\foreach \x in {0,1,...,8} {
    \draw (\x,0) -- (\x+1,0);
}

\node[circle, fill=black, inner sep=1pt] at (2,0.6) {};
\node[circle, fill=black, inner sep=1pt] at (4,0.6) {};
\node[circle, fill=black, inner sep=1pt] at (5,0.6) {};
\node[circle, fill=black, inner sep=1pt] at (6,0.6) {};

\node[circle, fill=black, inner sep=1pt] at (2,-0.6) {};
\node[circle, fill=black, inner sep=1pt] at (4,-0.6) {};
\node[circle, fill=black, inner sep=1pt] at (5,-0.6) {};
\node[circle, fill=black, inner sep=1pt] at (6,-0.6) {};

\draw (1,0) to[out=50, in=180, looseness=1] (2,0.6); 
\draw (2,0.6) to[out=0, in=-50, looseness=1]  (3,0);
\draw (1,0) to[out=-50, in=-180, looseness=1]  (2,-0.6);
\draw (2,-0.6) to[out=0, in=50, looseness=1]  (3,0);

\draw (3,0)  to[out=50, in=180, looseness=1] (4,0.6);
\draw (3,0)  to[out=-50, in=-180, looseness=1]  (4,-0.6);

\draw (6,0.6)  to[out=0, in=-50, looseness=1] (7,0);
\draw (6,-0.6)  to[out=0, in=50, looseness=1] (7,0);

\draw (4,0.6) -- (5,0.6);
\draw (5,0.6) -- (6,0.6);

\draw (4,-0.6) -- (5,-0.6);
\draw (5,-0.6) -- (6,-0.6);

\draw (0,0) to[out=-80, in=-100, looseness=2] (1,0); 
\draw (0,0) to[out=80, in=100, looseness=2] (1,0);  

\draw (7,0) to[out=-80, in=-100, looseness=2] (8,0); 
\draw (7,0) to[out=80, in=100, looseness=2] (8,0);  

\draw (8,0) -- (9,0); 
\node at (9.5,0) {$\cdots$}; 

\draw (8,0) to[out=45, in=180] (9,0.6); 
\node at (9.5,0.6) {$\cdots$}; 

\draw (8,0) to[out=-45, in=180] (9,-0.6); 
\node at (9.5,-0.6) {$\cdots$}; 

\node[below] at (0,0) {$k_0$};
\node[below] at (1,0) {$k_1$};
\node[below] at (3,0) {$k_2$};
\node[below] at (7,0) {$k_3$};
\node[below] at (8,0) {$k_4$};

\end{tikzpicture}

\end{center}

The graph $\Gamma_{\mathbf k,l}$ is a cactus graph only when $l=2$. For $l>2$ it is a graph with controlled cycles.

\end{example}

\begin{example}\label{example3}

Another example of graphs with controlled cycles is here. The length of each of the cycles sharing some common edges is the same. 

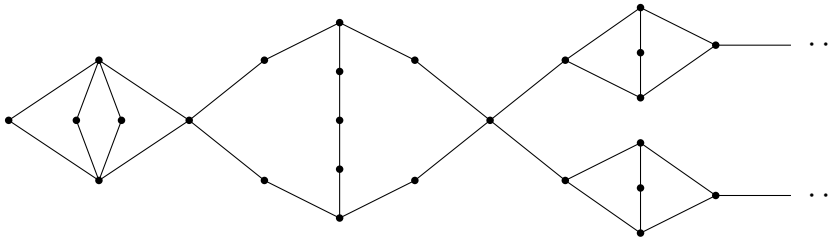
\begin{center}
\begin{tikzpicture}

\node[circle, fill=black, inner sep=1pt] at (-0.4,0) {};
\node[circle, fill=black, inner sep=1pt] at (0.5,0) {};
\node[circle, fill=black, inner sep=1pt] at (1.1,0) {};
\node[circle, fill=black, inner sep=1pt] at (2,0) {};
\node[circle, fill=black, inner sep=1pt] at (0.8,0.8) {};
\node[circle, fill=black, inner sep=1pt] at (0.8,-0.8) {};
\node[circle, fill=black, inner sep=1pt] at (3,0.8) {};
\node[circle, fill=black, inner sep=1pt] at (3,-0.8) {};
\node[circle, fill=black, inner sep=1pt] at (4,1.3) {};
\node[circle, fill=black, inner sep=1pt] at (4,0.65) {};
\node[circle, fill=black, inner sep=1pt] at (4,0) {};
\node[circle, fill=black, inner sep=1pt] at (4,-0.65) {};
\node[circle, fill=black, inner sep=1pt] at (4,-1.3) {};
\node[circle, fill=black, inner sep=1pt] at (5,0.8) {};
\node[circle, fill=black, inner sep=1pt] at (5,-0.8) {};
\node[circle, fill=black, inner sep=1pt] at (6,0) {};
\node[circle, fill=black, inner sep=1pt] at (7,0.8) {};
\node[circle, fill=black, inner sep=1pt] at (7,-0.8) {};
\node[circle, fill=black, inner sep=1pt] at (8,0.3) {};
\node[circle, fill=black, inner sep=1pt] at (8,0.9) {};
\node[circle, fill=black, inner sep=1pt] at (8,1.5) {};
\node[circle, fill=black, inner sep=1pt] at (8,-0.3) {};
\node[circle, fill=black, inner sep=1pt] at (8,-0.9) {};
\node[circle, fill=black, inner sep=1pt] at (8,-1.5) {};
\node[circle, fill=black, inner sep=1pt] at (9,1) {};
\node[circle, fill=black, inner sep=1pt] at (9,-1) {};

\draw (-0.4,0) -- (0.8,0.8);
\draw (-0.4,0) -- (0.8,-0.8);
\draw (1.1,0) -- (0.8,0.8);
\draw (1.1,0) -- (0.8,-0.8);
\draw (0.8,0.8) -- (2,0);
\draw (0.8,-0.8) -- (2,0);
\draw (0.5,0) -- (0.8,0.8);
\draw (0.5,0) -- (0.8,-0.8);

\draw (2,0) -- (3,0.8);
\draw (3,0.8) -- (4,1.3);
\draw (2,0) -- (3,-0.8);
\draw (3,-0.8) -- (4,-1.3);
\draw (4,-1.3) -- (4,1.3);
\draw (4,1.3) -- (5,0.8);
\draw (5,0.8) -- (6,0);
\draw (4,-1.3) -- (5,-0.8);
\draw (5,-0.8) -- (6,0);
\draw (6,0) -- (7,0.8);
\draw (6,0) -- (7,-0.8);
\draw (7,0.8) -- (8,0.3);
\draw (7,0.8) -- (8,1.5);
\draw (7,-0.8) -- (8,-1.5);
\draw (7,-0.8) -- (8,-0.3);
\draw (8,0.3) -- (8,1.5);
\draw (8,-0.3) -- (8,-1.5);
\draw (8,1.5) -- (9,1);
\draw (8,0.3) -- (9,1);
\draw (8,-0.3) -- (9,-1);
\draw (8,-1.5) -- (9,-1);
\draw (9,1) -- (10,1);
\draw (9,-1) -- (10,-1);

\node at (10.5,1) {$\cdots$};
\node at (10.5,-1) {$\cdots$};

\end{tikzpicture}

\end{center}
\end{example}

\begin{example}\label{example4}
Here is one more example of a graph with controlled cycles.

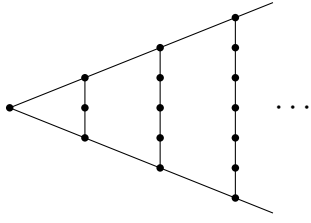
\begin{center}
\begin{tikzpicture}

\node[circle, fill=black, inner sep=1pt] at (0,0) {};
\node[circle, fill=black, inner sep=1pt] at (1,0) {};
\node[circle, fill=black, inner sep=1pt] at (1,0.4) {};
\node[circle, fill=black, inner sep=1pt] at (1,-0.4) {};
\node[circle, fill=black, inner sep=1pt] at (2,0) {};
\node[circle, fill=black, inner sep=1pt] at (2,0.4) {};
\node[circle, fill=black, inner sep=1pt] at (2,0.8) {};
\node[circle, fill=black, inner sep=1pt] at (2,-0.4) {};
\node[circle, fill=black, inner sep=1pt] at (2,-0.8) {};
\node[circle, fill=black, inner sep=1pt] at (3,0) {};
\node[circle, fill=black, inner sep=1pt] at (3,0.4) {};
\node[circle, fill=black, inner sep=1pt] at (3,0.8) {};
\node[circle, fill=black, inner sep=1pt] at (3,1.2) {};
\node[circle, fill=black, inner sep=1pt] at (3,-0.4) {};
\node[circle, fill=black, inner sep=1pt] at (3,-0.8) {};
\node[circle, fill=black, inner sep=1pt] at (3,-1.2) {};

\draw (0,0) -- (3.5,1.4);
\draw (0,0) -- (3.5,-1.4);
\draw (1,-0.4) -- (1,0.4);
\draw (2,-0.8) -- (2,0.8);
\draw (3,-1.2) -- (3,1.2);

\node at (3.8,0) {$\cdots$};

\end{tikzpicture}

\end{center}
\end{example}


A subgraph of $\Gamma$ is {\it $\Theta$-shaped} if there exist two vertices $v_1,v_2\in\Gamma$ and three simple paths $\gamma_1$, $\gamma_2$, $\gamma_3$ connecting $v_1$ with $v_2$ such that any two these paths have no common vertices, except $v_1$ and $v_2$. 

\begin{lem}
If $\Gamma$ is a graph with controlled cycles then
for any $\Theta$-shaped subgraph of $\Gamma$, the lengths of the two shortest paths out of three are equal.

\end{lem}
\begin{proof}
Let $\Gamma_0$ be a $\Theta$-shaped subgraph of $\Gamma$, and let $l_i$ be the length of $\gamma_i$, $i=1,2,3$. Then the lengths of the three cycles obtained from the three paths are $l_1+l_2$, $l_2+l_3$, $l_3+l_1$. The controlled cycles condition requires that the two maximal lengths of these cycles should be equal. Suppose that $l_1+l_2=l_2+l_3\geq l_3+l_1$. Then $l_2\geq l_1=l_3$.
\end{proof}





It would be interesting to find a graph-theoretic characterization of the class of graphs with controlled cycles.

\section{A metric on simple paths} 

Let $\gamma=(e_1,\ldots,e_n),\delta=(e'_1,\ldots,e'_m)\in S(\Gamma)$. Define $d(\gamma,\delta)=|\gamma\Delta\delta|$ as the number of edges in only one of the two paths.

Let $\Z$ denote the group of two elements, $0$ and $1$, with the standard metric $d(x,y)=|x-y|$, $x,y\in\Z$. We write $(\Z)^X$ for the set of maps on a set $X$ with values in $\Z$ with finite support, i.e. equal to 0 at all but finitely many points of $X$. For $a,b:X\to\Z$ set $\rho(a,b)=\sum_{x\in X}|a(x)-b(x)|$. If $a,b\in(\Z)^X$ then this sum is finite, and defines the Hamming metric on $(\Z)^X$ \cite{Hamming}. Note that $\rho$ is not proper, i.e. balls are not finite. It is known that for any proper left-invariant metric, $(\Z)^\infty$ has property A (\cite{Willett}, Proposition 2.3.3), but not for $\rho$, as $(\Z)^\infty$ contains an isometric copy of $\sqcup_{n\in\mathbb N}(\Z)^n$. 


For $\gamma=(e_1,\ldots,e_n)\in S(\Gamma)$, $e_1,\ldots,e_n\in E$, set 
$$
f(\gamma)(e)=\left\lbrace\begin{array}{cl}1 &\mbox{if\ }e\in\{e_1,\ldots,e_n\};\\0&\mbox{otherwise.}\end{array}\right.
$$ 
This gives a map $f:S(\Gamma)\to(\Z)^E$, which is injective by Lemma \ref{inject}. 

\begin{lem}
The map $f$ satisfies $d(\gamma,\delta)=\rho(f(\gamma),f(\delta))$ for any $\gamma,\delta\in S(\Gamma)$.


\end{lem}\label{lem0}
\begin{proof}
Let $\gamma,\delta\in S(\Gamma)$. 
If $e\in\gamma\cap\delta$ or $e\notin\gamma\cup\delta$ then $f(\gamma)(e)=f(\delta)(e)$, if $e\in\gamma\Delta\delta$ then $|f(\gamma)(e)-f(\delta)(e)|=1$. Summing up over all $e\in E$, we obtain $d(\gamma,\delta)=\rho(f(\gamma),f(\delta))$. 
\end{proof}

Lemma \ref{lem0} shows that $d$ is a metric on $S(\Gamma)$ and on any of its subsets, and that
we may think of $S(\Gamma)$ as a subset of $(\Z)^\infty$.

Let $P(\Gamma)=\sqcup_{v\in V}P_v(\Gamma)$ be the coarse union of these spaces with the metric $d$ on each component. Note that as the metric $d$ takes integer values, $S(\Gamma)$ and $P(\Gamma)$ are uniformly discrete.

\begin{example}
Let $\mathbf k=(1,2,3,\ldots)$, $l=2$. In the graph $\Gamma=\Gamma_{\mathbf k,2}$ from Example \ref{example}, each path from the root to a vertex $v_n$ such that $d_\Gamma(v_0,v_n)=n$ is determined by a 0-1 sequence of length $n$, as from each vertex the path may go to the right either by the upper, or by the lower edge (and, being simple, cannot go back to the left). Thus, there is a one-to-one correspondence $f_n:P_{v_n}(\Gamma)\to (\Z)^n$.
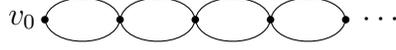
\begin{center}
\begin{tikzpicture}

\foreach \x in {0,1,...,4} {
    \node[circle, fill=black, inner sep=1pt] at (\x,0) {};
}
\foreach \x in {0,1,...,3} {
    \draw (\x,0) to[out=-80, in=-100, looseness=1] (\x+1,0);
}

\foreach \x in {0,1,...,3} {
    \draw (\x,0) to[out=80, in=100, looseness=1] (\x+1,0);
}

\node at (4.5,0) {$\cdots$}; 
\node at (-0.3,0) {$v_0$};

\end{tikzpicture}

\end{center}
As the length of each cycle is 2, one has $d(\gamma,\delta)=2\rho(f_n(\gamma),f_n(\delta))$ for any $\gamma,\delta\in P_{v_n}(\Gamma)$, hence $P(\Gamma)$ and $\sqcup_{n\in\mathbb N}(\Z)^n$ are isometric up to a constant factor.   

\end{example}

\begin{cor}
$S(\Gamma)$ and $P(\Gamma)$ are coarsely embeddable in a Hilbert space.

\end{cor}
\begin{proof}
The space $(\Z)^\infty$ is is isometric to the 0-1 sequences in $l_1$, which is coarsely embeddable in a Hilbert space (Example 5.2.3 and Theorem 5.3.1 in \cite{Nowak-Yu}), hence the same holds for its subsets and for their coarse unions.
\end{proof}

\begin{remark}
Example \ref{example4} shows that the sets $P_v(\Gamma)$ can be infinite. But if $\Gamma$ is a cactus graph then $P_v(\Gamma)$ is finite for any $v\in V(\Gamma)$.  

\end{remark}

\begin{remark}
The set $S(\Gamma)$ has a natural structure of a tree: its vertices are simple paths, and two simple paths are connected by an edge if they differ by an edge of $\Gamma$. But the graph metric of this tree differs from the metric $d$: some close points with respect to $d$ may be far with respect to the graph metric in this tree. 

\end{remark}

\section{Criterion for bounded geometry}

\begin{thm}\label{BG}
Let $\Gamma$ be a rooted graph with controlled cycles. 
Then the following are equivalent:
\begin{enumerate}
\item
the metric space $P(\Gamma)$ is of bounded geometry; 
\item
for any $n\in\mathbb N$ there there exists $m\in\mathbb N$ such that any $\gamma\in P$ intersects not more than $m$ cycles $C_1,\ldots,C_m$ of length $\leq n$. 
\end{enumerate}

\end{thm}
\begin{proof}
$(1)\Rightarrow (2)$. Suppose that $P(\Gamma)$ is of bounded geometry, and assume that there exists $n\in\mathbb N$ such that for any $m$ there exists $\gamma\in P(\Gamma)$ such that it intersects at least $m$ cycles $C_1,\ldots,C_m$ of length $\leq n$ each. Then $d(\gamma,\gamma_{C_i})\leq n$, so $\gamma_{C_1},\ldots,\gamma_{C_m}\in B_n(\gamma)$, i.e. $|B_n(\gamma)|\geq m$. As $m$ is arbitrary, this contradicts bounded geometry of $P(\Gamma)$.  

$(2)\Rightarrow (1)$. Let $n\in\mathbb N$. If (2) holds then there exists $m\in\mathbb Z$ such that any $\gamma\in P(\Gamma)$ intersects not more than $m$ cycles $C_i$ of length $\leq n$. Let $\gamma\in P_v(\Gamma)$ for some $v\in V(\Gamma)$, and let $\delta\in P_v(\Gamma)$ satisfy $d(\delta,\gamma)\leq n$, i.e. $\delta\in B_n(\gamma)\cap P_v(\Gamma)$. Then 
\begin{equation}\label{eq1}
\gamma\Delta\delta=C_{i_1}\cup\ldots\cup C_{i_k} 
\end{equation}
is the union of cycles, as the endpoints of $\gamma$ and $\delta$ coincide. Clearly, the length of each $C_{i_j}$ does not exceed $n$: otherwise it would contradict $d(\delta,\gamma)\leq n$. But there are not more than $m$ such cycles, so the cycles in (\ref{eq1}) belong to the set $C_1,\ldots,C_m$. Then 
$$
|\delta\Delta\gamma|\leq|C_1\cup\ldots\cup C_m|\leq m\cdot\max\nolimits_{i=1}^{m}|C_i|\leq mn. 
$$
Thus, $\delta$ may differ from $\gamma$ on no more than $mn$ edges, hence $|B_n(\gamma)\cap P_v(\Gamma)|\leq 2^{mn}$. 
As $v\in V(\Gamma)$ and $\gamma\in P_v(\Gamma)$ are arbitrary, we see that the size of all balls of radius $n$ in $P(\Gamma)$ is uniformly bounded.  
\end{proof}

\begin{example}\label{example2}
Let $\Gamma_{\mathbf k,l}$ be the graph from Example \ref{example}. Let $\mathbf k_1=(0,1,2,3,\ldots)$, $\mathbf k_2=(0,1,4,9,\ldots)$, $\Gamma_i=\Gamma_{\mathbf k_i,l}$, $i=1,2$. Then $P(\Gamma_1)$ does not have bounded geometry, and $P(\Gamma_2)$ has it. 

\end{example}

\begin{example}
If $\Gamma$ is a graph from Example \ref{example3} then it has bounded geometry iff the lengths of cycles go to infinity. The graph from Example \ref{example4}
has bounded geometry.

\end{example}

\section{Property A}

\begin{thm}\label{Thmain}
Let $\Gamma$ be a rooted graph with controlled cycles. 
Then the following are equivalent:
\begin{enumerate}
\item
$P(\Gamma)$  is of bounded geometry;
\item
$P(\Gamma)$   has property A;
\item
$P(\Gamma)$   satisfies the Higson--Roe condition.
\end{enumerate}

\end{thm}
\begin{proof}
$(1)\Rightarrow (2)$ Let $P(\Gamma)$ be of bounded geometry.  
Let $v\in V(\Gamma)$, $\gamma\in P_v(\Gamma)$, $n\in\mathbb N$. Let $A_\gamma(n)\subset P_v(\Gamma)$ be the set of all paths $\delta\in P_v(\Gamma)$ such that $\delta$ coincides with $\gamma$ everywhere except cycles of length $\leq n$. 

As $P(\Gamma)$ is of bounded geometry, by Theorem \ref{BG}, for any $n\in\mathbb N$ there there exists $m\in\mathbb N$ such that any $\gamma\in P_v(\Gamma)$ intersects not more than $m$ cycles $C_1,\ldots,C_m$ of length $\leq n$. If $\e\in A_\gamma(n)$ then $\e\in P_v(\gamma)$, and, hence, coincides with $\gamma$ on cycles of length $\leq n$, hence $d(\gamma,\e)\leq mn$. Thus, $A_\gamma(n)\subset B_{mn}(\gamma)$ for any $\gamma\in P(\Gamma)$, and item (2) of the definition of property A holds with $S=mn$. 

Let $\delta\in P_v(\Gamma)$, $d(\gamma,\delta)\leq R\in\mathbb N$. This means that $\gamma$ and $\delta$ cannot diverge on cycles of length $\geq R$. If $\e\in A_\gamma(n)$ then $\e$ and $\gamma$ do not diverge on cycles of length $>n$. Let $n\geq R$. Then, by transitivity of $\sim_n$, $\delta$ and $\e$ do not diverge on cycles of length $>n$, hence $\e\in A_\delta(n)$, therefore $A_\gamma(n)\Delta A_\delta(n)=\emptyset$, and item (1) of the definition of property A holds for $P(\Gamma)$ with $\epsilon=0$. 

$(2)\Rightarrow (3)$. This is Theorem 4.2.1 in \cite{Nowak-Yu}.

$(3)\Rightarrow(1)$. We shall show the contrapositive: unbounded geometry implies failure of the Higson--Roe condition. Suppose that $P(\Gamma)$ is not of bounded geometry. By Theorem \ref{BG}, there exists $n\in\mathbb N$ such that for any $m\in\mathbb N$ there exists $\gamma_m\in P(\Gamma)$ such that it intersects at least $m$ cycles $C_1,\ldots,C_m$ of length $\leq n$. If $\gamma_m\in P_{v_m}(\Gamma)$ for some $v_m\in V$ then $(\gamma_m)_{C_i}\in P_{v_m}(\Gamma)$, $i=1,\ldots,m$.

Let $X=\sqcup_{m\in\mathbb N}(\Z)^m$ be the coarse union of finite metric spaces $(\Z)^m$ with the Hamming metric $\rho$, and let $x=(x_1,\ldots,x_m)\in(\Z)^m$. 
Let $\delta_x\in P_{v_m}(\Gamma)$ be the path that coincides with $\gamma_m$ outside all the cycles $C_1,\ldots,C_m$, and within these cycles it coincides with $\gamma_m$ in $C_i$ if $x_i=0$, and it passes through the other side of the cycle $C_i$ if $x_i=1$. 
Set $f(x)=\delta_x$. This gives maps $f_m:(\Z)^m\to P_{v_m}(\Gamma)$ and a map $f:X\to P(\Gamma)$. If $x,y\in(\Z)^m$ then $2\rho(x,y)\leq d(f(x),f(y))\leq n\rho(x,y)$, hence $f$ is a coarse equivalence between $(X,\rho)$ and $(f(X),d)\subset P(\Gamma)$ (in fact, a bilipschitz equivalence).   

It was shown in \cite{Nowak} that $X$ not only does not have property A, but also does not satisfy the Higson--Roe condition. As $f(X)$ is coarsely equivalent to $X$, it also does not satisfy the Higson--Roe condition. By Proposition 4.2.5 in \cite{Nowak-Yu}, the Higson--Roe condition passes to subspaces (we don't know whether property A passes to subspaces without the bounded geometry assumption), hence $P(\Gamma)$ cannot satisfy the Higson--Roe condition. 
\end{proof}

It follows that for the graphs from Example \ref{example2}, $P(\Gamma_1)$ does not have property A, and $P(\Gamma_2)$  has it, $P(\Gamma)$ for the graphs from Example \ref{example3} have property A iff the lengths of cycles in $\Gamma$ go to infinity, and $P(\Gamma)$ for the graph from Example \ref{example4} has property A.

\section{Concluding remarks}


{\bf 1.} It follows from \cite{Ostrovsky} that if a space $P(\Gamma)$ is not of bounded geometry (and, hence, without property A) then it contains a subset of bounded geometry also without property A.  

\begin{thm}
Let $\Gamma$ be a rooted graph with controlled cycles and with uniformly bounded degrees. If $P(\Gamma)$ is not of bounded geometry then it contains a subset of bounded geometry without property A.
 
\end{thm}
\begin{proof} 
As any such $P(\Gamma)$ contains $X=\sqcup_{m\in\mathbb N}(\Z)^m$, it suffices to find a bounded geometry subset without property A in $X$. To this end we use \cite{Ostrovsky}. It is shown there that there exists a family of finite graphs $\Gamma_n$ such that 
\begin{enumerate}
\item $Y=\sqcup_{n\in\mathbb N}\Gamma_n$ has bounded geometry and does not have property A;
\item
there exists a bilipschitz map $f:\sqcup_{n\in\mathbb N}\Gamma_n\to X$
\end{enumerate}
(in fact, the map $f$ constructed in \cite{Ostrovsky} has range inside $l_1$, but it consists of 0-1 sequences, and the $l_1$ metric on this set is the same as the Hamming metric on $(\Z)^\infty$). As $f$ is a coarse equivalence, $f(Y)\subset X$ has bounded geometry and does not have property A.
\end{proof}

{\bf 2.} Note that simple paths cannot go back. To include more paths, we may consider the metric space $L(\Gamma)$ of simple loops in $\Gamma$ instead of the box spaces $P(\Gamma)$. It is possible to obtain the same results for $L(\Gamma)$ by allowing paths to pass vertices several times (in particular, loops should return back to the root), but each edge should still be passed not more than once. The construction of the inclusion of $L(\Gamma)$ into $(\Z)^\infty$ (and, hence, the metric on $L(\Gamma)$) should take into account orientation and duplicate all edges to distinguish cyclic paths going in opposite directions. All other arguments work for $L(\Gamma)$ with minor modifications (e.g. we have to take into account also cycles that have only a common vertex with a path). Therefore, Theorem \ref{Thmain} holds for $L(\Gamma)$ as well.

Unlike the space $L(\Gamma)$ of loops, the space $S(\Gamma)$ of all simple paths from the root without any restrictions on the endpoints is more complicated, and we don't know how to check property A for $S(\Gamma)$.   

{\bf 3.} Our choice of the metric $d$ on $S(\Gamma)$ is essential. For example, if we endow $S(\Gamma)$ with the Hausdorff distance metric $d^H$ then bounded geometry and property A fail to be equivalent. To see this, consider the graph $\Gamma=\Gamma_{\mathbf k,2}$ with $\mathbf k=(1,3,5,\ldots)$.
  
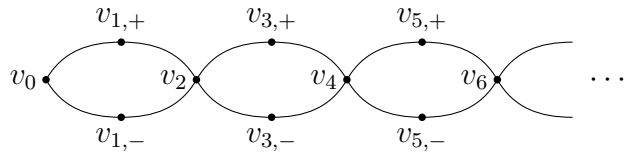
\begin{center}
\begin{tikzpicture}

\foreach \x in {0,2,...,6} {
    \node[circle, fill=black, inner sep=1pt] at (\x,0) {};
}
\foreach \x in {1,3,5} {
    \node[circle, fill=black, inner sep=1pt] at (\x,0.5) {};
}
\foreach \x in {1,3,5} {
    \node[circle, fill=black, inner sep=1pt] at (\x,-0.5) {};
}

\foreach \x in {0,2,...,6} {
    \draw (\x,0) to[out=60, in=180, looseness=1] (\x+1,0.5);
}
\foreach \x in {0,2,...,6} {
    \draw (\x,0) to[out=-60, in=-180, looseness=1] (\x+1,-0.5);
}

\foreach \x in {1,3,5} {
    \draw (\x,0.5) to[out=0, in=120, looseness=1] (\x+1,0);
}
\foreach \x in {1,3,5} {
    \draw (\x,-0.5) to[out=0, in=-120, looseness=1] (\x+1,0);
}

\node at (7.5,0) {$\cdots$}; 
\node at (-0.3,0) {$v_0$}; 
\node at (1,0.8) {$v_{1,+}$};
\node at (1,-0.8) {$v_{1,-}$};
\node at (1.7,0) {$v_2$}; 
\node at (3,0.8) {$v_{3,+}$};
\node at (3,-0.8) {$v_{3,-}$};
\node at (3.7,0) {$v_4$}; 
\node at (5,0.8) {$v_{5,+}$};
\node at (5,-0.8) {$v_{5,-}$};
\node at (5.7,0) {$v_6$};

\end{tikzpicture}

\end{center}
Let $\gamma,\delta\in P_v(\Gamma)$.
If $v=v_{2n}$, $n\in\mathbb N$, then $d^H(\gamma,\delta)\leq 1$. If $v=v_{{2n-1,\pm}}$ then $d^H(\gamma,\delta)\leq \frac{3}{2}$. Therefore, $P(\Gamma)$ with the Hausdorff metric is not of bounded geometry, but, $\diam P_v(\Gamma)\leq 3$ for any $v\in V(\Gamma)$ obviously implies property A for $P(\Gamma)$.

\section*{Acknowledgement}

The author is grateful to E. Troitsky for valuable comments.
 
This work was supported by the RSF grant 25-11-00018.


\end{document}